\documentclass[12pt, a4paper]{amsart}
\usepackage{amsmath, amssymb, amsthm, setspace, hyperref}

\newtheorem*{theorem}{Theorem}

\newtheorem*{lemma}{Lemma}
\newtheorem*{corollary}{Corollary}

\theoremstyle{definition}

\newtheorem*{ack}{Acknowledgments}

\begin{document}

\title{Jacobson's Refinement of Engel's Theorem for Leibniz Algebras}

\author[Bosko]{Lindsey Bosko}
\address{Department of Natural Sciences \& Mathematics, West Liberty University\\
West Liberty, WV 26074}
\email{lrbosko@ncsu.edu}

\author[Hedges]{Allison Hedges}
\thanks{The research of the second author was supported by NSF grant number 526797}
\address{Department of Mathematics, North Carolina State University\\
Raleigh, NC 27695}
\email{armcalis@ncsu.edu}

\author[Hird]{J.T. Hird}
\address{Department of Mathematics, North Carolina State University\\
Raleigh, NC 27695}
\email{johnthird@gmail.com}

\author[Schwartz]{Nathaniel Schwartz}
\thanks{The research of the fourth author was supported by NSF grant number 526797}
\address{Department of Mathematics, North Carolina State University\\
Raleigh, NC 27695}
\email{njschwar@ncsu.edu}

\author[Stagg]{Kristen Stagg}
\thanks{The research of the fifth author was supported by NSF grant number 526797}
\address{Department of Mathematics, University of Texas at Tyler\\
Tyler, TX 75799}
\email{klstagg@ncsu.edu}



\date{\today}

\begin{abstract}
We develop Jacobson's refinement of Engel's Theorem for Leibniz algebras. We then note some consequences of the result.
\end{abstract}

\maketitle

Since Leibniz algebras were introduced by Loday in \cite{loday} as a noncommutative generalization of Lie algebras, one theme is to extend Lie algebra results to Leibniz algebras. In particular, Engel's theorem has been extended in \cite{ayupov}, \cite{barnes2}, and \cite{patsourakos}. In \cite{barnes2}, the classical Engel's theorem is used to give a short proof of the result for Leibniz algebras. The proofs in \cite{ayupov} and \cite{patsourakos} do not use the classical theorem and, therefore, the Lie algebra result is included in the result. In this note, we give two proofs of the generalization to Leibniz algebras of Jacobson's refinement to Engel's theorem, a short proof which uses Jacobson's theorem and a second proof which does not use it. It is interesting to note that the technique of reducing the problem to the special Lie algebra case significantly shortens the proof for the general Leibniz algebras case. This approach has been used in a number of situations, see \cite{barnes}.  We also note some standard consequences of this theorem.  The proofs of the corollaries are exactly as in Lie algebras (see \cite{kaplansky}). Our result can be used to directly show that the sum of nilpotent ideals is nilpotent, and hence one has a nilpotent radical. In this paper, we consider only finite dimensional algebras and modules over a field $\mathbb{F}$.

An algebra $A$ is called Leibniz if it satisfies $x(yz) = (xy)z + y(xz)$. Denote by $R_a$ and $L_a$, respectively, right and left multiplication by $a \in A$. Then
\begin{align}
R_{bc} & = R_cR_b + L_bR_c\\
L_bR_c & = R_cL_b + R_{bc}\\
L_cL_b & = L_{cb} + L_bL_c.
\end{align}
Using (1) and (2) we obtain
\begin{align}
R_cR_b & = -R_cL_b.
\end{align}

It is known that $L_b = 0$ if $b = a^i, i\ge 2$, where $a^{1} = a$ and $a^{n}$ is defined inductively as $a^{n+1}=aa^{n}$.  Furthermore, for $n > 1$, $R_a^n = (-1)^{n-1}R_aL_a^{n-1}$. Therefore $R_a$ is nilpotent if $L_a$ is nilpotent. 

For any set $X$ in an algebra, we let $\langle \ X \ \rangle$ denote the algebra generated by $X$. Using (1), $R_{a^{2}} = (R_a)^{2} + L_aR_a$. Furthermore, the associative algebra generated by all $R_b$, $L_b$, $b \in \langle \ a \ \rangle$ is equal to $\langle \ R_a, L_a\ \rangle$. Suppose that $L_a^{n-1} = 0$. Then $R_a^n = 0$. For any $s \in \langle\ R_a, L_a\ \rangle$, $s^{2n -1}$ is a combination of terms with each term having at least $2n -1$ factors. Moreover, each of these factors is either $L_a$ or $R_a$. Any $L_a$ to the right of the first $R_a$ can be turned into an $R_a$ using (4). Hence, any term with $2n-1$ factors can be converted into a term with either $L_a$ in the first $n-1$ leading positions or $R_a$ in the last $n$ postitions. In either case, the term is 0 and $s^{2n-1} = 0$. Thus $\langle \ R_a, L_a\ \rangle$ is nil and hence nilpotent.

Let $M$ be an $A$-bimodule and let $T_a(m) = am$ and $S_a(m) = ma$, $a \in A$, $m \in M$. The analogues of (1) - (4) hold:
\begin{align}
S_{bc} & = S_cS_b + T_bS_c\\
T_bS_c & = S_cT_b + S_{bc}\\
T_cT_b & = T_{cb} + T_bT_c\\
S_cS_b & = - S_cT_b.
\end{align}
These operations have the same properties as $L_a$ and $R_a$, and the associative algebra $\langle \ T_a, S_a\ \rangle$ generated by all $T_b, S_b, b \in \langle \ a \ \rangle$ is nilpotent if $T_a$ is nilpotent. We record this as

\begin{lemma}
Let $A$ be a finite dimensional Leibniz algebra and let $a \in A$. Let $M$ be a finite dimensional $A$-bimodule such that $T_a$ is nilpotent on $M$. Then $S_a$ is nilpotent, and $\langle \ S_a, T_a\ \rangle$, the algebra generated by all $S_b, T_b$, $b \in \langle \ a \ \rangle$, is nilpotent.
\end{lemma}

A subset of $A$ which is closed under multiplication is called a Lie set.

\begin{theorem}
(Jacobson's refinement of Engel's theorem for Leibniz algebras) Let $A$ be a finite dimensional Leibniz algebra and $M$ be a finite dimensional $A$-bimodule. Let $C$ be a Lie set in $A$ such that $A = \langle \ C \ \rangle$. Suppose that $T_c$ is nilpotent for each $c \in C$. Then, for all $a \in A$, the associative algebra $B = \langle \ S_a, T_a\ \rangle$  is nilpotent. Consequently $B$ acts nilpotently on $M$, and there exists $m \in M$, $m \ne 0$, such that $am = ma = 0$ for all $a \in A$.
\end{theorem}

\begin{proof}

{\it Proof 1 (using the Lie result)}:
If $M$ is irreducible, then either $MA = 0$ or $ma = -am$ for all $a$ in $A$ and all $m$ in $M$ from Lemma 1.9 of \cite{barnes}. Since left multiplication of $A$ on $M$ gives a Lie module, the Jacobson refinement to Engel's theorem yields that $A$ acts nilpotently on $M$ on the left and hence on $M$ as a bimodule. If $M$ is not irreducible, then $A$ acts nilpotently on the irreducible factors in a composition series of $M$ and hence on $M$.  

{\it Proof 2 (independent of the Lie result)}: 
Let $x \in C$. Then $T_x$ is nilpotent and the associative algebra generated by $T_b$ and $S_b$ for all $b \in \langle \ x \ \rangle$ is nilpotent by the lemma. Since $\{a\ |\ aM = 0 = Ma\}$ is an ideal in $A$, we may assume that $A$ acts faithfully on $M$. 

Let $D$ be a Lie subset of $C$ such that $\langle\ D \ \rangle$ acts nilpotently on $M$, and $\langle\ D \ \rangle$ is maximal with these properties. If $C \subseteq \langle\ D \ \rangle$, then $A = \langle\ C \ \rangle = \langle\ D \ \rangle$, and we are done. Thus suppose that $C \nsubseteq \langle\ D \ \rangle$, and we will obtain a contradiction.

Let $E = \langle\ D \ \rangle \cap C$. $E$ is a Lie set since both $\langle\ D \ \rangle$ and $C$ are Lie sets. Since $D \subseteq \langle\ D \ \rangle$ and $D \subseteq C$, it follows that $D \subseteq E$ and $\langle\ D \ \rangle \subseteq \langle\ E \ \rangle$. Since $E \subseteq \langle\ D \ \rangle$, $\langle\ E \ \rangle \subseteq \langle\ D \ \rangle$ and $\langle\ D \ \rangle = \langle\ E \ \rangle$.

Let $\dim(M) = n$. Since $\langle\ D\ \rangle = \langle\ E\ \rangle$ acts nilpotently on $M$, $\sigma_1\cdots\sigma_n = 0$ where $\sigma_i = S_{d_i}$ or $T_{d_i}$ for $d_i \in E$. Then:

\noindent $\sigma_1\cdots\sigma_i\tau\sigma_{i+1}\cdots\sigma_{2n-1}=0$ where $\tau = S_a$ or $T_a$, $a \in A$, for all $i$.

If $x$ is any product in $A$ with $2n$ terms, of which $2n-1$ come from $E$, then $S_x$ and $T_x$ are linear combinations of elements as in the last paragraph. Hence $S_x = T_x = 0$, which implies that $x = 0$, since the representation is faithful.

There exists a smallest positive integer $j$ such that $\tau_1\cdots\tau_j C \subseteq \langle\ E\ \rangle$ for all $\tau_1, \dots ,\tau_j$ with $\tau_i = R_{d_i}$ or $L_{d_i}$ where $d_i \in E$. Then there exists an expression $z = \tau_{d_1}\cdots\tau_{d_{j-1}} x \notin \langle\ E\ \rangle$  for some $x \in C$ and $d_i \in E$. Note that $z \in C$ since $C$ is a Lie set. Consider $zE$. Now, $zC$, $Cz \subseteq C$ and $z \langle\ E\ \rangle$, $\langle\ E\ \rangle z \subseteq \langle\ E\ \rangle$. Therefore $zE$, $Ez \subseteq E$. Then $z^nE$, $Ez^n \subseteq E$ for all positive integers $n$, using induction and the defining identity for Leibniz algebras. Then $F = \{z^n, n \geq 1 \} \cup E$ is a Lie set contained in $C$, and since $z \notin \langle\ E\ \rangle$, it follows that $\langle\ E\ \rangle \subsetneq \langle\ F\ \rangle$.

It remains to show that $\langle\ F\ \rangle$ acts nilpotently on $M$. Define $M_0 = 0$ and $M_i = \{m \in M\ |\ Em, mE \subseteq M_{i -1}\}$. Since $E$ acts nilpotently on $M$, $M_k = M$ for some $k$. We show $zM_i$, $M_iz \subseteq M_i$. Clearly $zM_0 = M_0z = 0$. Suppose that $z$ acts invariantly on $M_i$ for all $i < t$. For $m \in M_t$, $d \in E$, $(zm)d = z(md) - m(zd) \in zM_{t-1} + mE \subseteq M_{t-1}$ with similar expressions for $(mz)d$, $d(mz)$ and $d(zm)$. Thus $z$ acts invariantly on each $M_i$, and hence $z^2$ does also. Thus $\langle\ z\ \rangle$ acts invariantly on each $M_i$. But $\langle\ z\ \rangle$ acts nilpotently on $M$ by the lemma. Therefore $F$ acts nilpotently on $M$ which is a contradiction.
\end{proof}

We obtain the abstract version of the theorem.

\begin{corollary}
Let $C$ be a Lie set in a Leibniz algebra $A$ such that $\langle\ C\ \rangle = A$ and $L_c$ is nilpotent for all $c \in C$. Then $A$ is nilpotent.
\end{corollary}

The following are extensions of results due to Jacobson, \cite{jacobson}, whose proofs are the same as in the Lie algebra case.

\begin{corollary}
If $T$ is an automorphism of $A$ of order $p$ and has no non-zero fixed points, then $A$ is nilpotent.
\end{corollary}

\begin{corollary}
If $D$ is a non-singular derivation of $A$ over a field of characteristic 0, then $A$ is nilpotent.
\end{corollary}

Finally, 

\begin{corollary}
If $B$ and $C$ are nilpotent ideals of $A$, then $B+C$ is a nilpotent ideal of $A$.
\end{corollary}

\vspace{1cm}

\begin{ack}
The authors completed this work as graduate students at North Carolina State University.  The authors wish to thank Professor E. L. Stitzinger for his guidance and support.
\end{ack}

\nocite{*}

\bibliographystyle{abbrv}
\bibliography{writeup}

\end{document}